\documentclass[11pt,a4paper]{article}
\usepackage{amsmath}
\usepackage{amsfonts}
\usepackage{amssymb}
\usepackage{amsthm}
\usepackage{amscd}
\usepackage{graphicx}
\newfont{\bb}{msbm10 at 10pt}
\def\rlopezr{\hbox{\bb R}}
\def\rlopezh{\hbox{\bb H}}
\def\rlopezs{\hbox{\bb S}}

\theoremstyle{plain}
\newtheorem{lemma}{Lemma}[section]

\newtheorem{theorem}{Theorem}[section]
\newtheorem{rlopezcorollary}{Corollary}[section]
\theoremstyle{definition}

\title{Parabolic surfaces in hyperbolic space with constant curvature}

\author{
    \emph{Rafael L\'opez}
    \thanks{Partially
supported by MEC-FEDER
 grant no. MTM2004-00109.} \\
    Departamento de Geometr\'{\i}a y Topolog\'{\i}a \\
    Universidad of Granada \\
    18071 Granada, Spain \\
    e-mail: \emph{rcamino@ugr.es} \\
    url: \emph{http://www.ugr.es/local/rcamino}
       }

\date{}

\begin{document}

\maketitle

\begin{abstract}
We study parabolic linear Weingarten surfaces in hyperbolic space
$\rlopezh^3$. In particular, we classify two family of parabolic
surfaces: surfaces with constant Gaussian curvature and surfaces
that satisfy the relation $a\kappa_1+b\kappa_2=c$, where $\kappa_i$
are the principal curvatures, and $a,b$ and $c$ are constant.

\end{abstract}

\section{Introduction}

Let $\rlopezh^3$ be the three-dimensional hyperbolic space. A
parabolic group of isometries of $\rlopezh^3$ is formed by
isometries that leave fix one double point of the ideal boundary
$\rlopezs^2_{\infty}$  of $\rlopezh^3$. We say that a surface $S$ is
a {\it parabolic surface} of $\rlopezh^3$ if it is invariant by a
group of parabolic isometries. A surface $S$ in  $\rlopezh^3$ is
called a {\it Weingarten surface} if there is some (smooth) relation
$W(\kappa_1,\kappa_2)=0$ between its two principal curvatures
$\kappa_1$ and $\kappa_2$.  In particular, if $K$ and $H$  denote
respectively  the Gauss curvature and the mean curvature of  $S$, we
have a relation $U(K,H)=0$.  In this note we study parabolic
Weingarten surfaces that satisfy the simplest case for $W$ and $U$,
that is,  of linear type:
\begin{equation}\label{rlopez:wein1}
a\ H+b\ K=c,
\end{equation}
and
\begin{equation}\label{rlopez:wein2}
a\ \kappa_1+b\ \kappa_2=c
\end{equation}
where $a,b,c\in\rlopezr$.  We say in both cases that $S$ is a {\it
linear} Weingarten surface.  In the set of linear Weingarten
surfaces, we mention  three families of surfaces that correspond
with trivial choices of the constants $a,b$ and $c$: surfaces with
constant Gauss curvature ($a=0$ in (\ref{rlopez:wein1})), surfaces
with constant mean curvature ($b=0$ in (\ref{rlopez:wein1}) or $a=b$
in (\ref{rlopez:wein2})) and umbilical surfaces ($a=-b$ and $c=0$ in
(\ref{rlopez:wein2})).
 Although these three kinds of surfaces have been studied
in the literature, the  classification of linear Weingarten surfaces
in the general case is almost completely open today.

A way  to seek linear Weingarten surfaces is focusing in rotational
surface because in such case, equations (\ref{rlopez:wein1}) and
(\ref{rlopez:wein2}) reduce into an ordinary differential equation.
In hyperbolic ambient, rotational linear Weingarten surfaces  have
been studied when the mean curvature is constant \cite{rlopez:cd},
in arbitrary dimension \cite{rlopez:le,rlopez:mo,rlopez:pa} or in
the spherical case \cite{rlopez:st,rlopez:yo}.

In this note we give a complete description and classification of
parabolic surfaces in $\rlopezh^3$ that satisfy equation
(\ref{rlopez:wein1}) when $a=0$ (constant Gaussian curvature) and
equation (\ref{rlopez:wein2}). A more detailed study can see in
\cite{rlopez:lo1} and \cite{rlopez:lo2}. Among the facts of our
interest, we ask whether the surface can be extended to be complete,
which it is given in terms of the generating curve, and whether the
surface is embedded.

\section{Preliminaries}

 Let us consider the
upper half-space model of the hyperbolic three-space $\rlopezh^3$,
namely,
$${\rlopezh}^{3}=:{\rlopezr}^3_+=\{(x,y,z)\in\rlopezr^3;z>0\}$$
equipped with the metric
$$\langle,\rangle=\frac{dx^{2}+dy^{2}+dz^{2}}{z^{2}}.$$
In what follows, we will  use the words  "vertical" or "horizontal"
in the usual affine sense of $\rlopezr^3_+$. The ideal boundary
$\rlopezs^2_{\infty}$ of $\rlopezh^3$ is
$\rlopezs^2_{\infty}=\{z=0\}\cup\{\infty\}$, the
one-compactification of the plane $\{z=0\}$.  The asymptotic
boundary of a set $\Sigma\subset\rlopezh^3$ is defined as
$\partial_{\infty}\Sigma=\overline{\Sigma}\cap \rlopezs^2_{\infty}$,
where $\overline{\Sigma}$ is the closure of $\Sigma$ in $\{z\geq
0\}\cup\{\infty\}$.

Let $G$ be a parabolic group  of isometries  of $\rlopezh^3$.
Without loss of generality, we  take the point  $\infty$ of
$\rlopezs^2_{\infty}$ as the point that fixes $G$.
 Then the group $G$ is defined by the horizontal (Euclidean) translations in the
direction of a  horizontal vector $\xi$ with $\xi\in\{z=0\}$.  The
space of orbits is  represented in any geodesic plane orthogonal to
$\xi$. Throughout this note,  we assume that  $\xi=(0,1,0)$.

 A
surface $S$ invariant by $G$ intersects  $P=\{(x,0,z);z>0\}$ in a
curve $\alpha$ called the {\it generating curve} of $S$. Consider
$\alpha(s)=(x(s),0,z(s))$   parametrized by the Euclidean
arc-length, $s\in I$ and $I$ an open interval including zero. Then
$x'(s)=\cos\theta(s)$ and $z'(s)=\sin\theta(s)$ for a certain
differentiable function $\theta$, where the derivative $\theta'(s)$
of the function $\theta(s)$ is the Euclidean curvature of $\alpha$.
A parametrization of $S$ is $X(s,t)=(x(s),t,z(s))$, $t\in\rlopezr$.
The principal curvatures $\kappa_i$ of $S$ are
\begin{equation}\label{rlopez:k1k2}
\kappa_1(s,t)=z(s)\theta'(s)+\cos\theta(s),\hspace*{1cm}\kappa_2(s,t)=\cos\theta(s),
\end{equation}
and the Gauss curvature $K$ is $K=\kappa_1\kappa_2-1$. Exactly
$\kappa_1$ is the hyperbolic curvature of the curve $\alpha$. Thus a
parabolic surface $S$ in $\rlopezh^3$ is given by a curve
$\alpha=(x(s),0,z(s))$ whose coordinate functions satisfy
\begin{equation}\label{rlopez:alpha}
 \left\{\begin{array}{lll}
 x'(s)&=& \displaystyle \cos\theta(s)\\
 z'(s)&=&\displaystyle \sin\theta(s)
\end{array}
\right.
\end{equation}
together the equation
\begin{equation}\label{rlopez:hk}
K=z(s)\cos\theta(s)\theta'(s)-\sin\theta(s)^2.
\end{equation}
if the Gaussian curvature $K$ is constant or
\begin{equation}\label{rlopez:wein22}
a z(s)\theta'(s)+(a+b)\cos\theta(s)=c
\end{equation}
 if $S$ satisfies the Weingarten relation
 (\ref{rlopez:wein2}). After an isometry of the ambient space formed by a
horizontal  translation orthogonal to $\xi$ followed by a
dilatation, we consider the initial conditions
\begin{equation}\label{rlopez:eq2}
x(0)=0,\hspace*{.5cm}z(0)=1,\hspace*{.5cm}\theta(0)=\theta_0.
\end{equation}

As a consequence of the uniqueness of solutions of an ordinary
differential equation, we have

\begin{lemma}\label{rlopez:si1}
Let $\alpha$ be a solution of the initial value problem
(\ref{rlopez:alpha})-(\ref{rlopez:hk}) or
(\ref{rlopez:alpha})-(\ref{rlopez:wein22}). Let $s_0\in I$.
\begin{enumerate}
\item If $z'(s_0)=0$, then $\alpha$
is symmetric with respect to the vertical line $x=x(s_0)$ of the
$xz$-plane.
\item If $\theta'(s_0)=0$, then  $\alpha$ is a
straight-line.
\end{enumerate}
\end{lemma}

\section{Parabolic surfaces with constant Gaussian curvature}\label{rlopez:sec3}

 Let us assume that $S$ is a parabolic surface in $\rlopezh^3$ with
 constant Gauss curvature $K$. Then the generating curve $\alpha$
 satisfies (\ref{rlopez:alpha})-(\ref{rlopez:hk}).
 Consider $z'(s)$ as a function of the new variable $z(s)$. If we
put $p=z'$ and $x=z$, we have  $x p(x)p'(x)=K+p(x)^2$. Setting
$y=p^2$, we write $x y'(x)=2K+2y(x)$. The solutions of this equation
are $y(x)=K x^2-K$, that is,
\begin{equation}\label{rlopez:z2}
z'(s)^2=K(z(s)^2-1).
\end{equation}
A new differentiation in (\ref{rlopez:z2}) gives $ z''(s)= Kz(s)$,
whose solutions are  well known. With respect to the function
$x(s)$,   we   express $x(s)$ in terms of an elliptic integral from
the equality $x(s)=\int_0^s\sqrt{1-z'(t)^2}\ dt$.

\begin{enumerate}
\item Case $K>0$. The solution  is $z(s)=\cosh{(\sqrt{K}s)}$
whose domain is  $(-s_1,s_1)$ with
$$s_1=\frac{1}{\sqrt{K}} \mbox{arcsinh}(\frac{1}{\sqrt{K}}).$$
Moreover,  the behaviour of $\alpha$ at  the ends points of
$(-s_1,s_1)$ is
$$\lim_{s\rightarrow s_1}z(s_1)=\sqrt{\frac{1+K}{K}}\hspace*{1cm}
\lim_{s\rightarrow s_1}z'(s_1)=1.$$ The height of $S$, that is, the
hyperbolic distance between the horospheres at heights $z=z(s_1)$
and $z=z_0=1$  is
$$\frac12\log\left(\frac{K+1}{K}\right).$$

\item Case $K=0$. The solution is $\alpha(s)=(s,0,1)$, that is,
$\alpha$ is a horizontal straight-line and the surface is a
horosphere.

\item Case $K<0$. The solution is $z(s)=\cos\left(\sqrt{-K}
s\right)$.  Depending on the value of $K$, the generating curve
$\alpha$ meets $\rlopezs^2_{\infty}$. If $-1\leq K<0$, $\alpha$
intersects $\rlopezs^2_{\infty}$ making an angle such that
$\sin\theta_1=\sqrt{-K}$. The domain of $\alpha$ is
$(-\pi/2,\pi/2)$. In the particular case that $K=-1$, $\alpha$ is a
halfcircle that orthogonally meets $\rlopezs^2_{\infty}$. If $K<-1$,
$S$ is not complete and the curve $\alpha$ is a graph on an interval
of  $\rlopezs^2_{\infty}$. The parameter $s$ goes in the range
$(-\frac{1}{\sqrt{-K}}\mbox{arcsin}(\frac{1}{\sqrt{-K}}),\frac{1}{\sqrt{-K}}\mbox{arcsin}(\frac{1}{\sqrt{-K}})).$
Analogously as in the case $K>0$, the  height of the surface is
$$\frac12\log\left(\frac{K-1}{K}\right).$$

\end{enumerate}

\begin{theorem}
Let $\alpha$ be the generating curve of a parabolic surface $S$ in
hyperbolic space $\rlopezh^3$ with constant Gauss curvature $K$,
where $\alpha$ is  the solution of
(\ref{rlopez:alpha})-(\ref{rlopez:hk}).
 Assume that the initial velocity of $\alpha$ is a horizontal vector.
 Then we have:
\begin{enumerate}
\item Case $K>0$. The curve $\alpha$ is convex with exactly one minimum and
 it is a graph on $\rlopezs^2_{\infty}$ defined in some bounded interval
 $I=(-x_1,x_1)$. See Figure \ref{rlopez:fig1}, (a).
\item Case $K=0$. The curve $\alpha$ is a horizontal straight-line and $S$ is a
horosphere. See Figure \ref{rlopez:fig1}, (b).
\item Case $K<0$. The curve $\alpha$ is concave with exactly one maximum and
it  is a graph on $\rlopezs^2_{\infty}$ defined in some bounded
interval $I=(-x_1,x_1)$. If $-1\leq K<0$, the curve $\alpha$ meets
$\rlopezs^2_{\infty}$ making an angle $\theta_1$ with
$\sin\theta_1=\sqrt{-K}$. See Figure \ref{rlopez:fig2} (a). If
$K<-1$, $\alpha$ does not intersect $\rlopezs^2_{\infty}$ and at the
end points, the curve is vertical. See Figure \ref{rlopez:fig2} (b).

\end{enumerate}
In cases 1) and 3), the height of $S$ is $\frac12\log\left(\frac{K+1}{K}\right)$ and
$\frac12\log\left(\frac{K-1}{K}\right)$ respectively.

\end{theorem}

\begin{figure}[htbp]\begin{center}
\includegraphics[width=5cm]{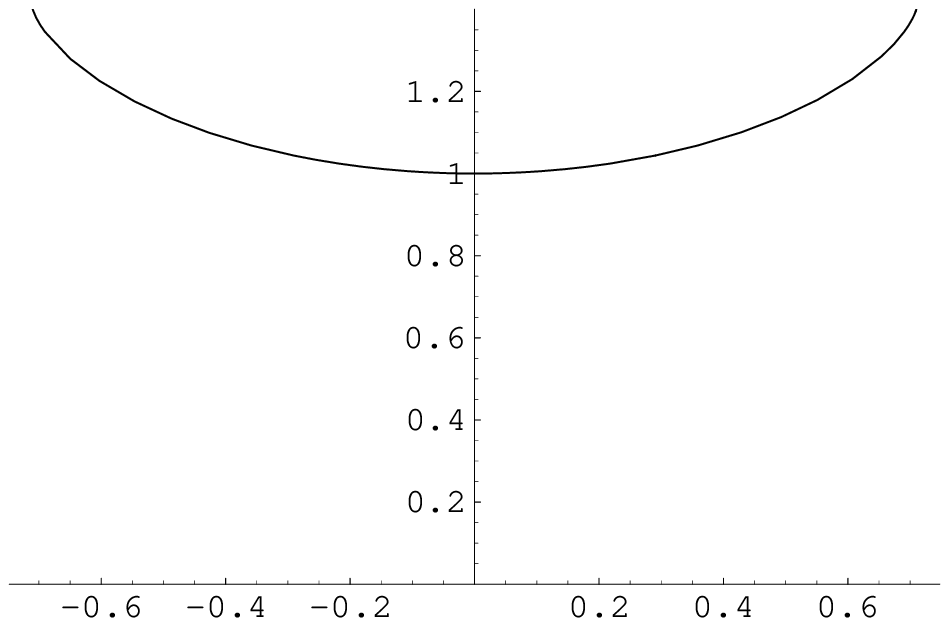}\includegraphics[width=5cm]{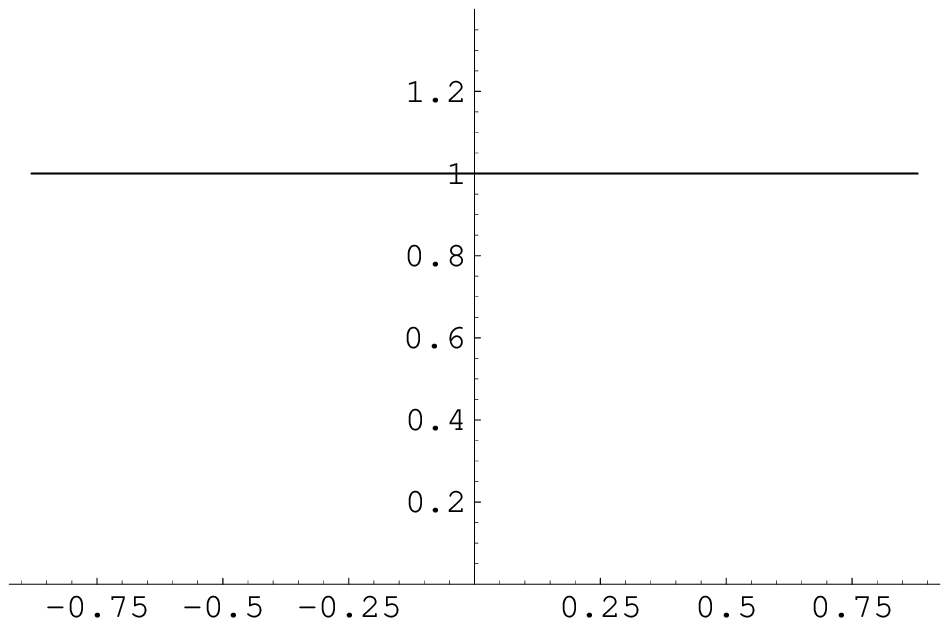}
\end{center}
\begin{center}(a)\hspace*{4cm}(b)\end{center}
\caption{The generating curves of parabolic surfaces with constant
Gaussian curvature $K$. The initial angle is $\theta(0)=0$. Case
(a): $K=1$; Case (b): $K=0$.}\label{rlopez:fig1}
\end{figure}

\begin{figure}[htbp]\begin{center}
\includegraphics[width=5cm]{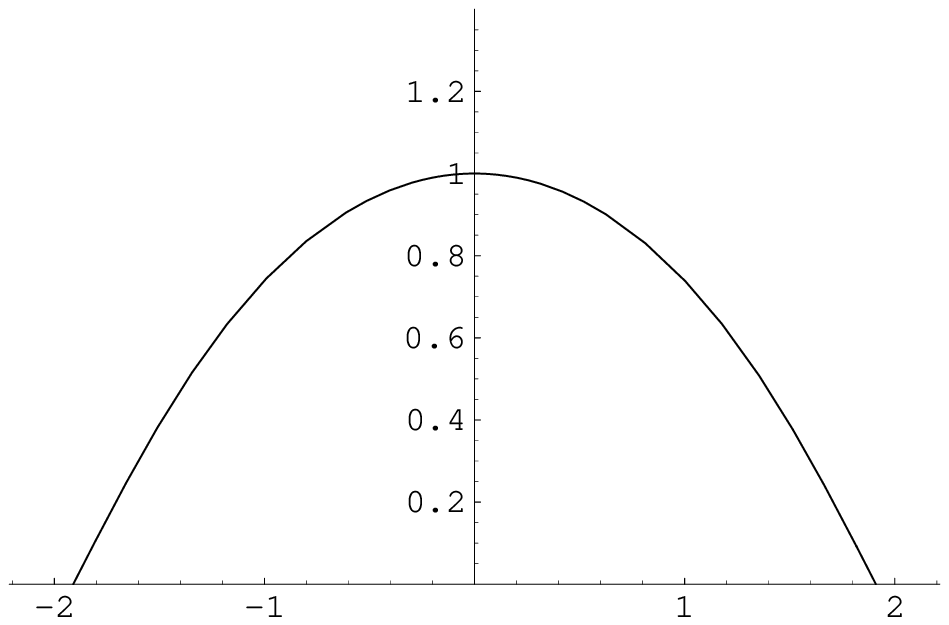}\includegraphics[width=5cm]{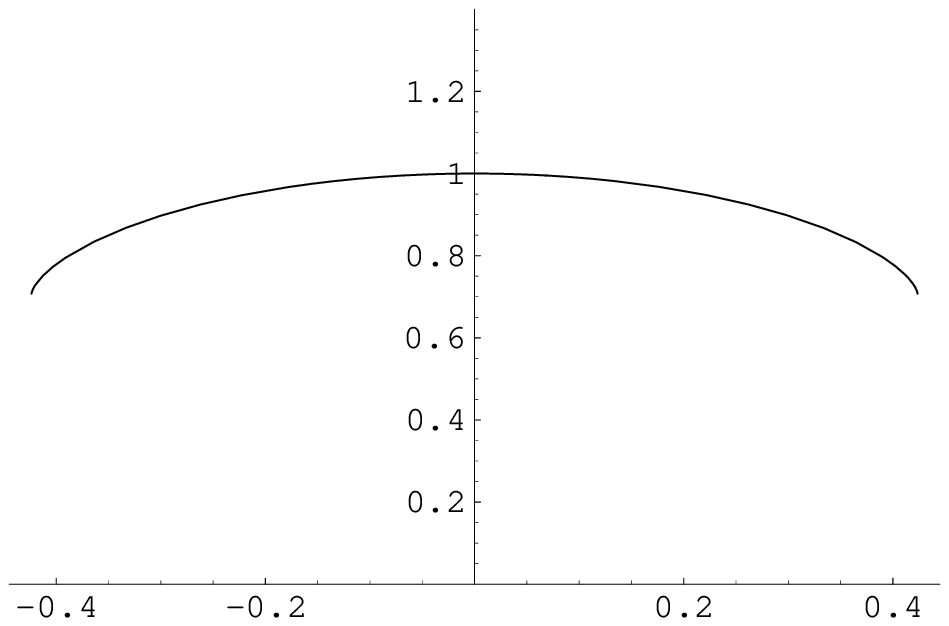}
\end{center}
\begin{center}(a)\hspace*{4cm}(b)\end{center}
\caption{The generating curves of parabolic surfaces with constant
Gaussian curvature $K$. The initial angle is $\theta(0)=0$. Case
(a): $K=-0.5$; Case (b): $K=-2$.}\label{rlopez:fig2}
\end{figure}

\begin{rlopezcorollary}For each number $K$ with $-1\leq K<0$, there exists a non-umbilical
complete parabolic surface in $\rlopezh^3$ with constant Gauss
curvature $K$. For these surfaces, the asymptotic boundary is formed
by two circles tangent at  the point fixed by the group of parabolic
isometries.
\end{rlopezcorollary}

\begin{theorem}
 Any non-umbilical parabolic surface in $\rlopezh^3$ with constant Gaussian curvature $K$ with
$K<-1$ or $K\geq 0$ and with a horizontal tangent plane is not
complete. Moreover, its asymptotic boundary is the point fixed by
the group of parabolic isometries.
\end{theorem}

 Finally, we remark that if we want to have the complete classification of
 parabolic surfaces with constant Gaussian curvature, we must  change the
 starting angle $\theta_0$ in (\ref{rlopez:eq2}) in order to obtain
 all such surfaces. See \cite{rlopez:lo1}. In the range of value $K$, with $K\in (1,0)$,
 there exist non
complete parabolic surfaces and the asymptotic boundary of each such
surface is a circle of $\rlopezs^2_{\infty}$. In Figure
 \ref{rlopez:fig3},  we show two such parabolic surfaces with
 $\theta_0=\pi/4$. As conclusion of our study, we have

\begin{theorem}
 Any non-umbilical parabolic surface in $\rlopezh^3$ with constant Gaussian curvature $K$ with
$K<-1$ or $K\geq 0$ is not complete. Moreover, its asymptotic
boundary is the point fixed by the group of parabolic isometries.
\end{theorem}

\begin{rlopezcorollary} Any parabolic surface immersed in hyperbolic space
$\rlopezh^3$ with constant Gaussian curvature is a graph on
$\rlopezs^2_{\infty}$. In particular, it is embedded.
\end{rlopezcorollary}

\begin{figure}[htbp]\begin{center}
\includegraphics[width=5cm]{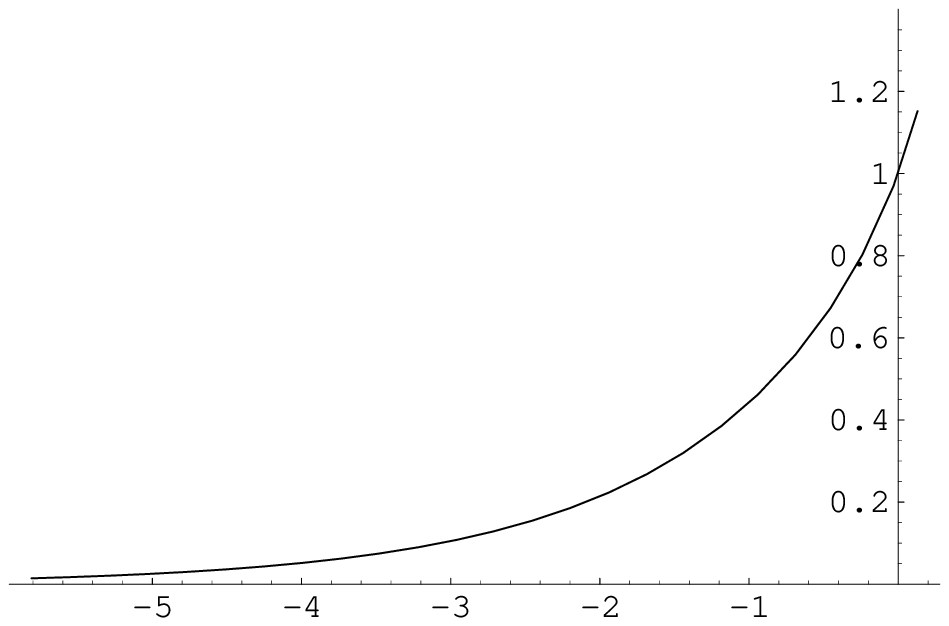}\includegraphics[width=5cm]{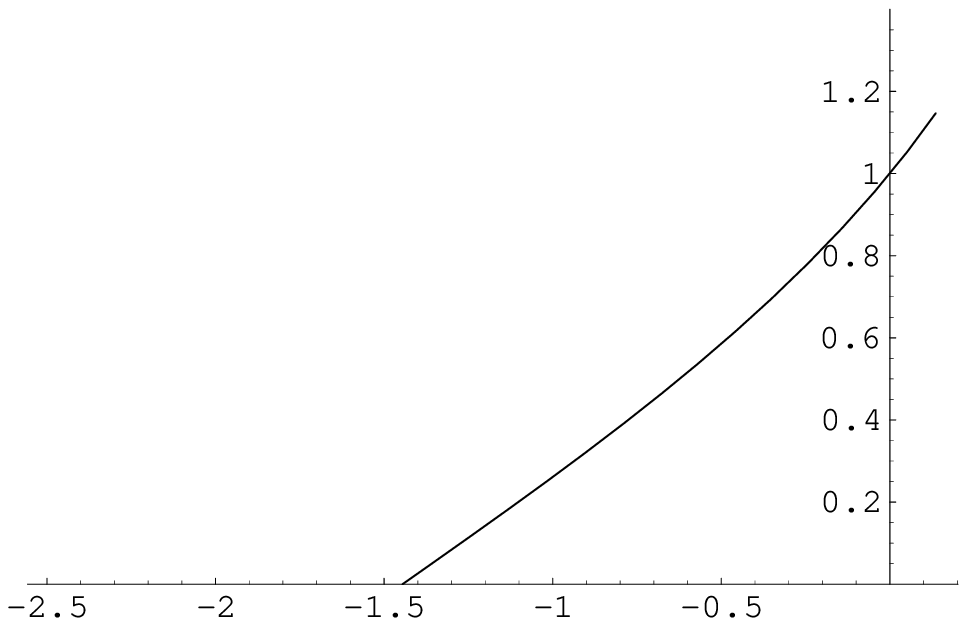}
\end{center}
\begin{center}(a)\hspace*{4cm}(b)\end{center}
\caption{The generating curves of parabolic surfaces with constant
Gaussian curvature $K$. The initial angle is $\theta(0)=\pi/4$. Case
(a): $K=0$; Case (b): $K=-1/4$.}\label{rlopez:fig3}
\end{figure}

%

\section{Linear Weingarten parabolic surfaces}
In this section we shall consider parabolic surfaces that satisfy
the relation $a\kappa_1+b\kappa_2=c$. In the case that $a$ or $b$ is
zero, that is,  that one of the principal curvatures $\kappa_i$ is
constant, we have

\begin{theorem}
 The only parabolic surfaces in $\rlopezh^3$ with one constant principal curvature
are totally geodesic planes, equidistant surfaces, horospheres and
Euclidean horizontal right-cylinders.
\end{theorem}

\begin{proof} We use (\ref{rlopez:k1k2}). If $\kappa_1=c$, then $\theta'(s)
z(s)=c-\cos\theta(s)$. By differentiation of this expression and
using (\ref{rlopez:alpha}) we obtain $\theta''(s)=0$ for all $s$.
Then $\theta'$ is constant and hence that  from the Euclidean
viewpoint, the curve is a piece of a straight-line or a circle. If
$\kappa_2$ is constant, then  $\cos\theta(s)=c$ and this  means that
$\theta$ is constant. Thus  $\alpha$ is a straight-line.
\end{proof}

We write  the general case (\ref{rlopez:wein2})  as
\begin{equation}\label{rlopez:w1} \kappa_1=m\kappa_2+n
\end{equation}
where $m,n\in\rlopezr$, $m\not=0$. By using (\ref{rlopez:k1k2}),
Equation (\ref{rlopez:w1}) writes as
\begin{equation}\label{rlopez:w11}
\theta'(s)=\frac{(m-1)\cos\theta(s)+n}{z(s)}.
\end{equation}
After a change of orientation on the surface, we suppose in our
study that $n\geq 0$. We discard the trivial cases of Weingarten
surfaces, that is,  $(m,n)=(1,0)$ and   $m=-1$. We consider that the
starting angle $\theta_0$ in (\ref{rlopez:eq2}) is $\theta_0=0$.
Equation (\ref{rlopez:w11}) yields at $s=0$, $\theta'(0)=n+m-1$. By
Lemma \ref{rlopez:si1}, if $\theta'(0)\not=0$, then $\theta(s)$ is a
 monotonic function on $s$. Let $(-\bar{s},\bar{s})$ be the
maximal domain of solutions of
(\ref{rlopez:alpha})-(\ref{rlopez:w11}) under the initial conditions
(\ref{rlopez:eq2}) and denote
$\bar{\theta}=\lim_{s\rightarrow\bar{s}}\theta(s)$. Depending on the
sign of $\theta'(0)$, we consider three cases.

\subsection{Case $n+m-1>0$}\label{rlopez:subsec1}

As $\theta'(0)>0$,  $\theta(s)$ is a strictly increasing
function.

\begin{enumerate}
\item Subcase $m<n+1$. In particular, $n>0$.
 We prove that $\theta$ attains the value $\pi/2$.
Assume on the contrary that $\bar{\theta}\leq\pi/2$ and we will
arrive to a contradiction. As $z'(s)=\sin\theta(s)>0$,  $z(s)$ is
strictly increasing in $(0,\bar{s})$. Then $z(s)\geq z_0$ and  the
derivatives of $\{x(s),z(s),\theta(s)\}$ in equations
(\ref{rlopez:alpha})-(\ref{rlopez:w11}) are bounded. This means that
$\bar{s}=\infty$. As
$\lim_{s\rightarrow\infty}z'(s)=\sin\bar{\theta}>0$, then
$\lim_{s\rightarrow\infty}z(s)=\infty$.  Multiplying in
(\ref{rlopez:w11}) by $\sin\theta$ and integrating, we obtain
\begin{equation}\label{rlopez:integral2}
n+\cos\theta(s)=\frac{2-m}{z(s)}\int_0^s
\left(\sin\theta(t)\cos\theta(t)\right)\ dt+\frac{n+1}{z(s)}.
\end{equation}
Let $s\rightarrow\infty$ in (\ref{rlopez:integral2}). If the
integral that appears in (\ref{rlopez:integral2}) is bounded, then
$n+\cos\bar{\theta}=0$, that is, $\cos\bar{\theta}=n=0$:
contradiction. If the integral is not bounded, and using the
L'H\^{o}pital's rule, $n+\cos\bar{\theta}=(2-m)\cos\bar{\theta}$,
that is, $(m-1)\cos\bar{\theta}+n=0$. Then $m-1\leq 0$ and the
hypothesis $n+m-1>0$ yields $\cos\bar{\theta}= n/(1-m)>1$:
contradiction.

Therefore, there exists a first value $s_1$ such that
$\theta(s_1)=\pi/2$.  We prove that $\theta(s)$ attains the value
$\pi$. By contradiction, we assume
  $\bar{\theta}\leq\pi$ and  $z(s)$
is strictly increasing again. We then have $\bar{s}=\infty$ again
and $\theta'(s)\rightarrow 0$ as $s\rightarrow\infty$. If $z(s)$ is
bounded, then (\ref{rlopez:integral2}) implies
$(m-1)\cos\bar{\theta}+n=0$. As $m-1=n=0$ is  impossible, then
$m-1>0$ since $\cos\bar{\theta}<0$. But the hypothesis $m<n+1$
implies that $\cos\bar{\theta}=-n/(m-1)<-1$, which it is a
contradiction. Thus $z(s)\rightarrow\infty$ as $s\rightarrow\infty$.
By using (\ref{rlopez:integral2}) again, and letting
$s\rightarrow\infty$, we have $n+\cos\bar{\theta}=0$. In particular,
$0<m<2$. We obtain a second integral from (\ref{rlopez:w11})
multiplying by $\cos\theta(s)$:
$$\sin\theta(s)=\frac{s}{z(s)}+\frac{1}{z(s)}
\int_0^s\left( n\cos\theta(t)+(m-2)\cos^2\theta(t)\right)\ dt.$$ If
the integral is bounded, then $\sin^2\bar{\theta}=1$: contradiction.
Thus, the integral  is not bounded and L'H\^{o}pital rule implies
$\sin^2\bar{\theta}=1+n\cos\bar{\theta}+(m-2)\cos^2\bar{\theta}$.
This equation, together $n+\cos\bar{\theta}=0$ yields
$(m-2)\cos^2\bar{\theta}=0$: contradiction.

As conclusion,  there exists a first value $s_2$ such that
$\theta(s_2)=\pi$. By Lemma \ref{rlopez:si1}, the curve $\alpha$ is
symmetric with respect to the line $x=x(s_2)$.  By symmetry,
$\alpha$ is invariant by a group of horizontal translations
orthogonal to the orbits of the parabolic group.

\item Subcase $m\geq n+1$. With this hypothesis and as
$\theta'(s)>0$, Equation (\ref{rlopez:w11})
 implies that $\cos\theta(s)\not=-1$ for any $s$. Thus $-\pi <\theta(s)<\pi$.
For $s>0$, $z'(s)=\sin\theta(s)>0$ and then $z(s)$ is increasing on
$s$ and so, $\theta'(s)$ is a bounded function. This implies
$\bar{s}=\infty$.
 We show that either there exists $s_0>0$ such $\theta(s_0)=\pi/2$ or
$\lim_{s\rightarrow\infty}\theta(s)=\pi/2$.

As in the above subcase, and with the same notation,  if
$\theta(s)<\pi/2$ for any $s$, then $n+\cos\bar{\theta}=0$ or
$(m-1)\cos\bar{\theta}+n=0$.  As $\cos\bar{\theta}\geq 0$ and since
$m-1\geq n$, it implies that this occurs if and only if $n=0$ and
$\bar{\theta}=\pi/2$. In such case,
$z''(s)=\theta'(s)\cos\theta(s)>0$, that is, $z(s)$ is a convex
function.
 As conclusion, if
$n>0$, there exists a value  $s_0$ such that $\theta(s_0)=\pi/2$,
and there exists $\bar{\theta}\in (\pi/2,\pi]$ such that
$\lim_{s\rightarrow\infty}\theta(s)=\bar{\theta}$.
\end{enumerate}

\begin{theorem}\label{th1} Let $\alpha(s)=(x(s),0,z(s))$ be the generating curve of a
parabolic surface  $S$ in  $\rlopezh^3$ whose principal curvatures
satisfy the relation $\kappa_1=m\kappa_2+n$. Consider $n\geq 0$ and
that $\theta(0)=0$ in the initial condition (\ref{rlopez:eq2}).
Assume $n+m-1>0$.
\begin{enumerate}
\item Case $m<n+1$. Then $\alpha$ is invariant by a group of translations in the
$x$-direction. Moreover,  $\alpha$ has self-intersections and it
presents one maximum and one minimum in each period, with vertical
points  between maximum and minimum. See Figure \ref{rlopez:fig4}
(a).
\item Case $m\geq n+1$.
If $n>0$, then $\alpha$ has a minimum with self-intersections. See
Figure \ref{rlopez:fig4} (b). If $n=0$, then $\alpha$ is a convex
graph on $\rlopezs^2_{\infty}$, with a minimum. See Figure
\ref{rlopez:fig5} (a).
\end{enumerate}
\end{theorem}


\subsection{Case $n+m-1=0$}\label{rlopez:subsec2}


In the case that $n+m-1=0$ where $\theta'(0)=0$, and by Lemma
\ref{rlopez:si1}, $\theta(s)=0$ for any $s$.

\begin{theorem} Let $\alpha(s)=(x(s),0,z(s))$ be the generating curve of a
parabolic surface  $S$ in  $\rlopezh^3$. Assume that the principal
curvatures of $S$ satisfy the relation $\kappa_1=m\kappa_2+n$ with
$n+m-1=0$ and $n\geq 0$. If $\theta(0)=0$ in the initial condition
(\ref{rlopez:eq2}), then $S$ is  a horosphere.
\end{theorem}


\subsection{Case $n+m-1<0$}


If $n+m-1<0$, $\theta(s)$ is a decreasing function. As $n\geq 0$ and
from (\ref{rlopez:w11}), $\cos\theta(s)\not=0$. This implies that
$\theta(s)$ is a bounded function with $-\pi/2<\theta(s)<\pi/2$. If
$\bar{s}=\infty$ and as $z(s)>0$, then both functions $\theta'(s)$
and $z'(s)$ go to $0$ as $s\rightarrow\infty$. By (\ref{rlopez:eq2})
and (\ref{rlopez:w11}), we have $(m-1)\cos\bar{\theta}+n=0$ and
$\sin\bar{\theta}=0$: contradiction. This proves that
$\bar{s}<\infty$.

As consequence, $z(s)\rightarrow 0$ since on the contrary,
$\theta'(s)$ would be bounded and $\bar{s}=\infty$. We now use
(\ref{rlopez:integral2}). Letting $s\rightarrow\bar{s}$ and by
L'H\^{o}pital rule again, we obtain $(m-1)\cos\bar{\theta}+n=0$,
that is,  $\cos\bar{\theta}\geq -n/(m-1)$. Finally,
$z''(s)=\theta'(s)\cos\theta(s)<0$, that is, $\alpha$ is concave.

\begin{theorem} \label{rlopez:th2} Let $\alpha(s)=(x(s),0,z(s))$
be the generating curve of a parabolic surface  $S$ in  $\rlopezh^3$
whose principal curvatures satisfy the relation
$\kappa_1=m\kappa_2+n$. Consider $n\geq 0$ and that  $\theta(0)=0$
in the initial condition (\ref{rlopez:eq2}). Assume $n+m-1<0$.  Then
$\alpha$  is a concave graph on  $\rlopezs^2_{\infty}$ with one
maximum and it intersects $\rlopezs^2_{\infty}$ with a contact angle
$\bar{\theta}$, $\cos\bar{\theta}= -n/(m-1)$. See Figure
\ref{rlopez:fig5} (b).
 \end{theorem}

\begin{figure}[htbp]\begin{center}
\includegraphics[width=5cm]{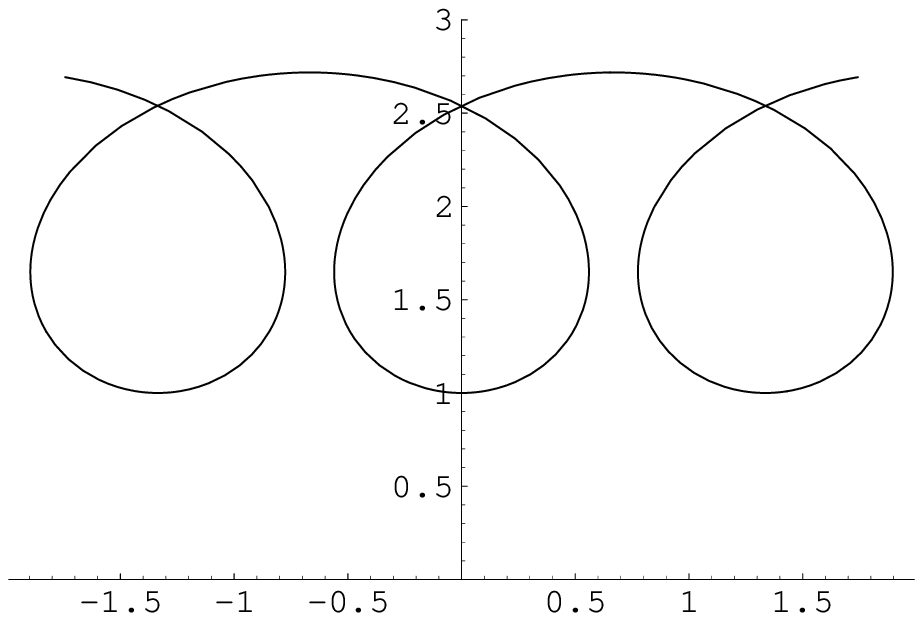}\hspace*{2cm}
\includegraphics[width=5cm]{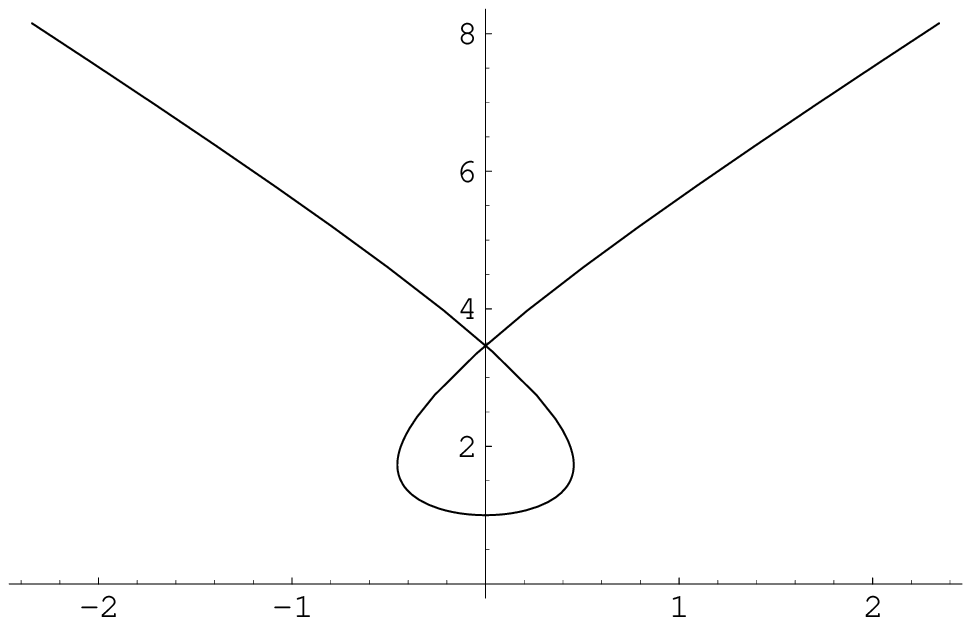}
\end{center}
\caption{The generating curves of a parabolic surfaces with
$\kappa_1=m\kappa_2+n$ and $n+m-1>0$. We consider in (a) the subcase
 $m<n+1$, with $m=1$ and $n=2$. In (b) we show the subcase $m\geq n+1$
 with $m=3$ and $n=1$. }\label{rlopez:fig4}
\end{figure}

\begin{figure}[htbp]\begin{center}
\includegraphics[width=5cm]{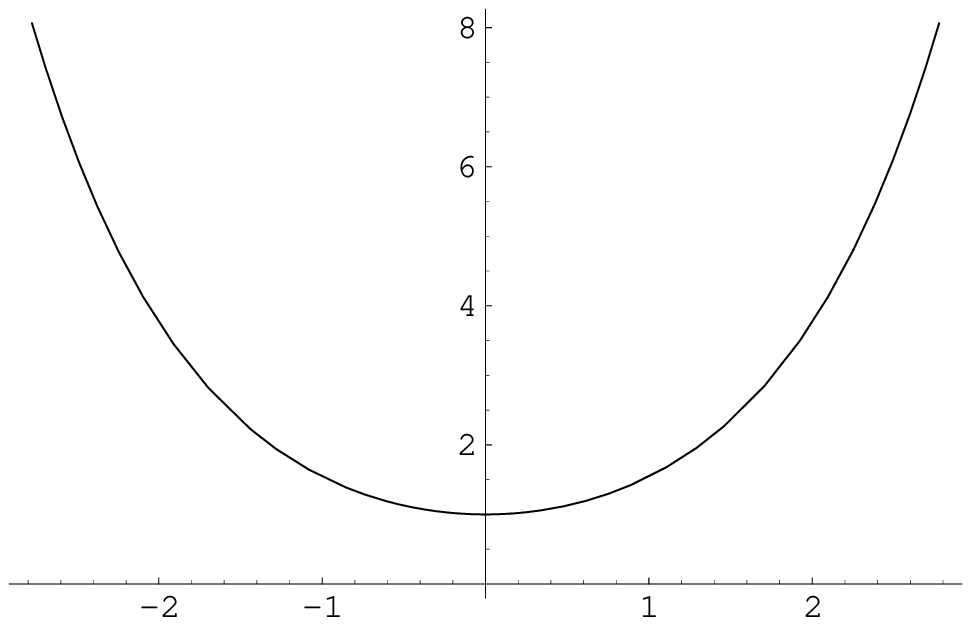}\hspace*{2cm}
\includegraphics[width=5cm]{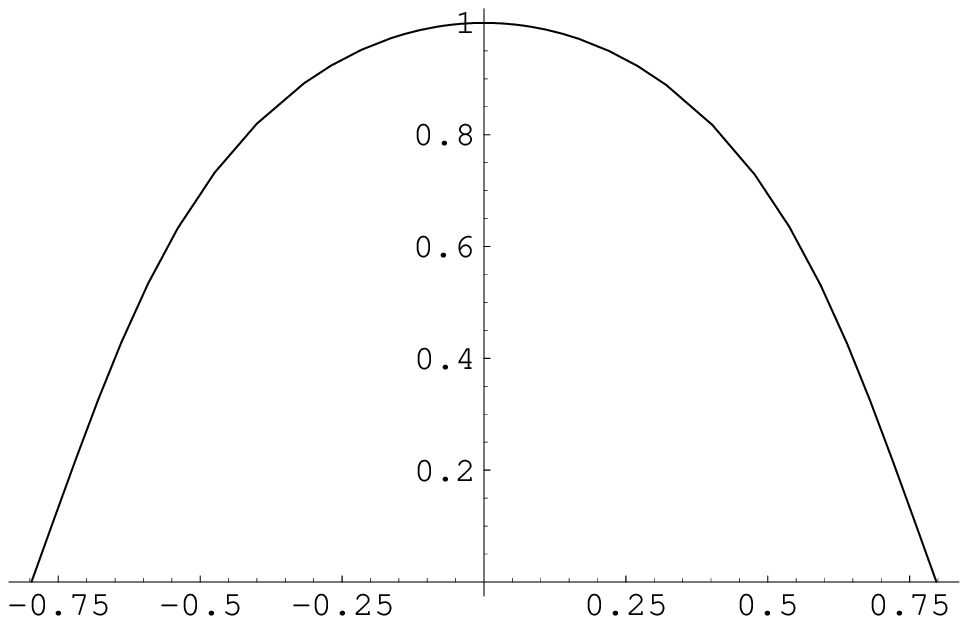}
\end{center}
\begin{center}(a)\hspace*{8cm}(b)\end{center}
\caption{The generating curves of a parabolic surfaces with
$\kappa_1=m\kappa_2+n$. We consider in (a) the case $n+m-1>0$ and
subcase $m\geq n+1$, with $m=2$ and $n=0$. In (b), we show the case
$n+m-1<0$ with $m=-2$ and $n=1$.
 }\label{rlopez:fig5}
\end{figure}

As it as  pointed out in the above Section \ref{rlopez:sec3}, the
classification of the parabolic surfaces in $\rlopezh^3$ that
satisfy the relation $\kappa_2=m\kappa_1+n$ finishes when we go
changing the initial angle $\theta_0$ in (\ref{rlopez:eq2}) in the
range $0\leq\theta_0\leq 2\pi$. For example, in the case studied in
subsection \ref{rlopez:subsec1}, that is, $n+m-1>0$, and subcase
$m<n+1$, the velocity vector $\alpha'(s)$ takes all values of the
interval $[0,2\pi]$. Thus, and using the uniqueness of solutions of
an ordinary differential equation, the case $\theta_0=0$  covers all
possibilities. In this way, we would have to consider all cases. As
an example, we focus in the case  of subsection
\ref{rlopez:subsec2}. We omit the proof.

\begin{theorem} Let $\alpha(s)=(x(s),0,z(s))$ be the generating curve of a parabolic
surface  $S$ in hyperbolic space $\rlopezh^3$. Assume that the
principal curvatures of $S$ satisfy the relation
$\kappa_1=m\kappa_2+n$ with $n+m-1=0$. If $\theta(0)\in (0,2\pi)$ in
the initial condition (\ref{rlopez:eq2}), then $\alpha$ is a curve
with self-intersections, with  one maximum and asymptotic to
$\rlopezs^2_{\infty}$ at infinity, that is,
 $\lim_{s\rightarrow\pm\infty}z(s)=0$.
 See Fig. \ref{rlopez:fig6}.

\end{theorem}

\begin{figure}[htbp]\begin{center}
\includegraphics[width=7cm]{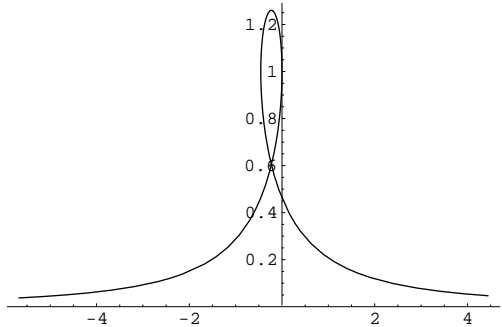}
\end{center}
\caption{The generating curve of a parabolic surface with
$\kappa_1=m\kappa_2+n$ and $n+m-1=0$. Here $m=-2$ and $n=3$. The
starting angle $\theta_0$ is $\theta_0=\pi/2$.}\label{rlopez:fig6}
\end{figure}

\end{document}